\documentclass{amsart}

\usepackage{amssymb}

\newtheorem{theorem}{Theorem}[section]
\newtheorem{proposition}[theorem]{Proposition}
\newtheorem{lemma}[theorem]{Lemma}
\newtheorem{corollary}[theorem]{Corollary}

\newcommand \CC     {C_\square}
\newcommand \Down   {\operatorname{Down}}
\newcommand \E      {\mathbb E}
\newcommand \GR[1]  {\left \lvert #1 \right \rvert}
\newcommand \Irr    {\operatorname{Irr}}
\newcommand \OC[1]  {\Delta \left( #1 \right)}
\newcommand \R      {\mathbb R}
\newcommand \height {\operatorname{ht}}
\newcommand \lk     {\operatorname{lk}}
\newcommand \supp   {\operatorname{supp}}

\renewcommand \phi {\varphi}

\title[]
{A criterion for a locally distributive semilattice
to have CAT(0) orthoscheme complex}
\author{TOUNAI Shouta}
\address{Graduate School of Mathematical Sciences, The University of Tokyo,
3-8-1 Komaba, Meguro, Tokyo, 153-8914 Japan}
\email{tounai@ms.u-tokyo.ac.jp}

\begin{document}

\begin{abstract}
  In this paper, we give a simple criterion
  for a locally distributive semilattice
  to have CAT(0) orthoscheme complex.
  Namely, the orthoscheme complex of a locally distributive semilattice $S$
  is CAT(0) if and only if $S$ is a flag semilattice,
  that is, any pairwise bounded triple of $S$ is bounded.
\end{abstract}

\maketitle

\section{Introduction}

Gromov \cite{Gro} showed that a cubical complex has non-positive curvature
if and only if the link of each vertex is a flag complex.
This theorem has a lot of applications.
A typical example is the proof that any right angled Artin group
is a CAT(0) group, which goes as follows (see \cite{CD} for more details).
For a right angled Artin group $A_\Gamma$, one can construct
a cubical complex $\mathcal S_\Gamma$ with fundamental group $A_\Gamma$,
which is called the Salvetti complex associated to $A_\Gamma$.
Using Gromov's characterization, one can check that
$\mathcal S_\Gamma$ have non-positive curvature.
Thus $A_\Gamma$ acts properly, cocompactly by isometries on the universal cover
of $\mathcal S_\Gamma$, which is a CAT(0) geodesic space.

It is, however, still open whether all Artin groups are CAT(0) groups.
Brady and McCammond \cite{BM} introduced orthoscheme complexes
as a generalization of cubical complexes.
An orthoscheme is a Euclidean simplex which appears
in the barycentric subdivision of the cube $[-1, 1]^n$.
The orthoscheme complex of a graded poset $P$ is a piecewise Euclidean complex
obtained by gluing orthoschemes along the chains of $P$.
A precise definition will be given in Section~\ref{sec:ortho}.
Brady and McCammond showed the following.
\begin{enumerate}
  \item If the orthoscheme complex of the poset $NPC_n$
    of the non-crossing partitions is a CAT(0) space,
    then the $n$-string braid group is a CAT(0) group.
  \item For $n \le 5$, the orthoscheme complex of $NPC_n$
    is a CAT(0) space.
\end{enumerate}
Thus the $n$-string braid group is a CAT(0) group for $n \le 5$.
They conjectured that (2) holds for arbitrary $n$.
Haettel, Kielak and Schwer showed that (2) holds for $n \le 6$ \cite{HKS}.

Now, it seems to be important to develop criteria
for a graded poset to have CAT(0) orthoscheme complex.
Chalopin et al. \cite{CCHO} established some sufficient conditions.
For example, they showed the following.
\begin{enumerate}
  \item The orthoscheme complex of a modular lattice is a CAT(0) space.
  \item The orthoscheme complex of a locally distributive flag semilattice
    is a CAT(0) space.
\end{enumerate}
Relevance between the CAT(0) properties of orthoscheme complexes
and the computational complexity of the 0-extension problem
was pointed out (see \cite{CCHO} for more details).

It seems, however, that there were few necessary and sufficient conditions
for a graded poset to have CAT(0) orthoscheme complex.
In this paper, we discuss a translation and an extension
of Gromov's characterization for orthoscheme complexes.
We say a semilattice $S$ is a \emph{flag semilattice}
if any pairwise bounded triple of $S$ is again bounded.
As a translation, we show that
the orthoscheme complex of a locally Boolean semilattice $S$ is a CAT(0) space
if and only if $S$ is a flag semilattice (Theorem~\ref{thm:LBS-CAT0}).
As an extension, we show that
the orthoscheme complex of a locally distributive semilattice $S$
is a CAT(0) space if and only if $S$ is a flag semilattice
(Theorem~\ref{thm:LDS-CAT0}).
We also show that the orthoscheme complex
of any locally distributive semilattice
can be embedded in that of some locally Boolean semilattice as a convex subset.

The rest of this paper is organized as follows.
In Section~\ref{sec:pre}, we introduce some notion and terminology.
In Section~\ref{sec:rep}, we establish a representation theorem for
locally distributive semilattices.
In Section~\ref{sec:met}, we review some notion concerning
CAT(0) geodesic spaces and Euclidean polyhedral complexes.
In Section~\ref{sec:ortho}, we discuss an extension
of Gromov's characterization for orthoscheme complexes.

\section{Preliminaries} \label{sec:pre}

\subsection{Simplicial complexes}

An \emph{abstract simplicial complex} $K$ is a family of finite sets
such that any subset of any element of $K$ is again an element of $K$.
An element of $K$ is said to be a \emph{face} of $K$,
and an element of a face of $K$ is said to be a \emph{vertex} of $K$.
In our definition, the empty set is a face of $K$ unless $K = \emptyset$.
Let $V(K)$ denote the set of the vertices of $K$.

\subsection*{Simplicial maps}
Let $K$ and $L$ be abstract simplicial complexes.
A \emph{simplicial map} from $K$ to $L$ is a map $f \colon V(K) \to V(L)$
such that the image $f(\sigma)$ of any face $\sigma$ of $K$ is a face of $L$.
A simplicial map $f$ is an \emph{isomorphism} if $f$ is bijective
and the inverse $f^{-1}$ is also a simplicial map from $L$ to $K$.
If an isomorphism between $K$ and $L$ exists,
then $K$ and $L$ is said to be \emph{isomorphism},
and we write $K \cong L$.

\subsection*{Simplices}
Let $\sigma$ be a finite set.
The abstract simplicial complex consisting of all subsets of $\sigma$
is called the \emph{simplex} of vertex set $\sigma$,
which will be denoted by $\widetilde \sigma$.

\subsection*{Joins}
Let $K$ and $L$ be abstract simplicial complexes.
For simplicity, we assume that $V(K)$ and $V(L)$ are disjoint.
Otherwise we replace $v \in V(K)$ with $(1, v)$,
and $w \in V(L)$ with $(2, w)$.
The \emph{join} of $K$ and $L$ is defined  by
\begin{equation*}
  K * L = \{\, \sigma \cup \tau \;|\; \sigma \in K, \ \tau \in L \,\}.
\end{equation*}
The vertex set $V(K * L)$ is given by the disjoint union $V(K) \sqcup V(L)$.
The inclusions induce simplicial maps $K \hookrightarrow K * L$
and $L \hookrightarrow K * L$.

\subsection*{Links}
Let $K$ be an abstract simplicial complexes,
and $\sigma$ a face of $K$.
The \emph{link} of $\sigma$ in $K$ is defined by
\begin{equation*}
  \lk(\sigma; K) = \{\, \tau \in K \;|\; \sigma \cap \tau = \emptyset, \
  \sigma \cup \tau \in K \,\}.
\end{equation*}
The link $\lk(\emptyset, K)$ of the empty face
is the same as $K$ itself.
The link $\lk(\{ v \}; K)$ of a $0$-face is simply denoted by $\lk(v; K)$.
If $\tau$ is a face of $\lk(\sigma; K)$, then the iterated link
$\lk(\tau; \lk(\sigma; K))$ coincides with $\lk(\sigma \cup \tau; K)$.

\subsection*{Flag complexes}
An abstract simplicial complex $K$ is said to be a \emph{flag complex} if
the following condition holds for any finite subset $\sigma$ of vertices:
if any two-element subset of $\sigma$ forms a face of $K$,
then $\sigma$ itself is also a face of $K$.

\begin{proposition} \label{prp:flag}
  The following hold.
  \begin{enumerate}
    \item An abstract simplicial complex $K$ is a flag complex
      if and only if the following hold
      for any faces $\sigma_1, \sigma_2, \sigma_3$ of $K$:
      if all pairwise unions $\sigma_1 \cup \sigma_2$, $\sigma_1 \cup \sigma_3$
      and $\sigma_2 \cup \sigma_3$ are faces of $K$,
      then $\sigma_1 \cup \sigma_2 \cup \sigma$ is also a face of $K$.
    \item For any finite set $\sigma$,
      the simplex $\widetilde \sigma$ is a flag complex.
    \item If an abstract simplicial complex $K$ is a flag complex,
      then the link $\lk(\sigma; K)$ is a flag complex
      for any face $\sigma$ of $K$.
    \item For abstract simplicial complexes $K$ and $L$,
      the join $K * L$ is a flag complex if and only if
      both $K$ and $L$ are flag complexes.
  \end{enumerate}
\end{proposition}

\begin{proof}
  The proof is straightforward.
\end{proof}

\subsection*{Geometric realizations}
For a finite set $\sigma$, the \emph{standard simplex}
of vertex set $\sigma$ is defined by
\begin{equation*}
  \Delta^\sigma = \Bigl\{\, \sum_{v \in \sigma} t_v v \;\Bigm|\;
  t_v \ge 0, \ \sum_{v \in \sigma} t_v = 1 \,\Bigr\}
  \subset \R^{(\sigma)},
\end{equation*}
where $\R^{(\sigma)}$ denotes the free linear space
$\bigoplus_{v \in \sigma} \R \sigma$ with basis $\sigma$.
Geometrically, $\Delta^\sigma$ is a point if $\dim \sigma = 0$,
a segment if $\dim \sigma = 1$, and a triangle if $\dim \sigma = 2$.
For an abstract simplicial complex $K$, the \emph{geometric realization}
of $K$ is defined by
\begin{equation*}
  \GR K = \bigcup_{\sigma \in K} \Delta^\sigma \subset \R^{(V(K))}.
\end{equation*}
Equivalently, $\GR K$ can be defined by
\begin{equation*}
  \GR K = \Bigl\{\, x = \sum_{v \in V(K)} t_v v \;\Bigm|\;
    t_v \ge 0, \
    \sum_{v \in V(K)} t_v = 1, \
    \supp x \in K
  \,\Bigr\},
\end{equation*}
where $\supp x = \{\, v \in V(K) \;|\;  t_v \neq 0 \,\}$.
Usually, we consider $\GR K$ as a topological space with the weak topology
with respect to $\Delta^\sigma$ for $\sigma \in K$, that is,
the coarsest topology on $\GR K$ such that the inclusion
$\Delta^\sigma \hookrightarrow \GR K$ is a continuous map
for each $\sigma \in K$.
In this paper, we consider a piecewise Euclidean metric
on geometric realizations, and study their curvature properties.
Such a metric defines another topology on the geometric realization.
This topology coincides with the weak topology
if and only if $K$ is locally finite.

\subsection{Partially ordered sets}

A \emph{partially ordered set} (\emph{poset} for short)
is a pair of a set $P$ and a partial order $\le$ on $P$.
We denote a poset $(P, {\le})$ simply by $P$ if no confusion can arise.
Let $S$ be a subset of a poset $P$.
Then $S$ can be seen as a poset by the restriction of the partial order on $P$.
In this case, $S$ is said to be a \emph{induced subposet} of $P$.

Let $P = (P, {\le_P})$ and $Q = (Q, {\le_Q})$ be posets.
A map $f \colon P \to Q$ is \emph{order preserving} if $x \le_P y$ implies
$f(x) \le_Q f(y)$ for any $x, y \in P$.
We say $f$ is \emph{strictly order preserving}
if $x <_P y$ implies $f(x) <_Q f(y)$ for any $x, y \in P$.

Let $P$ be a poset.
A \emph{chain} of $P$ is a totally ordered subset of $P$.
The \emph{length} of a chain $C$ is defined to be $\# C - 1$.
The \emph{height} $\height(P)$ of $P$ is defined to be
the least upper bound of the lengths of all chains of $P$,
which might be $\infty$.
The \emph{height} $\height_P(x)$ of an element $x$ of $P$
is defined to be $\height(P^{\le x})$.
We say that $P$ has \emph{locally finite height}
if the height of any elements of $P$ is finite.
Let us note that $P$ has finite height if and only if the order complex of $P$
is finite dimensional.

A subset $A$ of $P$ is said to be \emph{bounded above},
or simply \emph{bounded}, if there exists $u \in P$
such that $A \subset P^{\le u}$.

\subsection*{Lattices}
A \emph{lattice} is a poset $L$ such that any pair $x, y \in L$ has
the greatest lower bound and the least upper bound,
which will be denoted by $x \wedge y$ and $x \vee y$, respectively.
We say $L$ is \emph{modular} if the modular law
\begin{equation*}
  (x \vee y) \wedge z = x \vee (y \wedge z)
\end{equation*}
holds for any $x, y, z \in L$ with $x \le z$.
We say $L$ is \emph{distributive} if the distributive law
\begin{equation*}
  x \wedge (y \vee z) = (x \wedge y) \vee (x \wedge z)
\end{equation*}
holds for any $x, y, z \in L$.
We say $L$ is \emph{bounded} if $L$ has a minimum and a maximum,
which will denoted by $0$ and $1$, respectively.
A bounded lattice $L$ is said to be \emph{complemented} if for any $x \in L$
there exists $y \in L$ such that $x \wedge y = 0$ and $x \vee y = 1$.
A complemented distributive lattice is called a \emph{Boolean} lattice.

\subsection*{Semilattices}
A \emph{meet-semilattices}, or simply \emph{semilattice}, is a poset $S$
such that any pair $x, y \in S$ has the greatest lower bound,
which will be denoted by $x \wedge y$.

\begin{lemma}
  Let $S$ be a semilattice of locally finite height,
  and $A$ a non-empty subset of $S$ closed under $\wedge$,
  that is, $x, y \in A$ implies $x \wedge y \in A$.
  Then $A$ has a minimum element.
\end{lemma}

\begin{proof}
  First, we show that $A$ has a minimal element.
  If $A$ has no minimal elements, then we can take
  an infinite strictly decreasing sequence $a_0 > a_1 > \dotsb$ of $A$.
  Thus we obtain
  \begin{equation*}
    \infty > \height_S(a_0) > \height_S(a_1) > \dotsb,
  \end{equation*}
  but $\height_S(x)$ is non-negative for any $x \in S$,
  which is a contradiction.

  Thus $A$ has a minimal element $m$.
  Then $m$ is the minimum element of $A$,
  since we have $x \ge x \wedge m = m$ for $x \in A$.
\end{proof}

\begin{proposition}
  Let $S$ be a non-empty semilattice of locally finite height.
  Then the following hold.
  \begin{enumerate}
    \item $S$ has a minimum elements, which will be denoted by $0$.
    \item Any bounded pair of $S$ has the least upper bound.
    \item $S^{\le x}$ is a bounded lattice for any $x \in S$,
  \end{enumerate}
\end{proposition}

\begin{proof}
  To show (1), apply the previous lemma to $S$ itself.
  To show (2), similarly consider $S^{\le x} \cap S^{\le y}$.
  (3) follows from (1) and (2).
\end{proof}

If $x, y \in S$ are bounded, we denote their least upper bound by $x \vee y$.
We can see $\vee$ as a partial binary operator on $S$.
For a property $(\Phi)$ for bounded lattices,
we say $S$ is locally $(\Phi)$ if $S^{\le x}$ satisfies $(\Phi)$
for any $x \in S$.
For example, a locally distributive semilattice is a semilattice $S$
such that $S^{\le x}$ is a distributive lattice for any $x \in S$.
We say that $S$ is a \emph{flag semilattice} if
any pairwise bounded triple of elements of $S$ is bounded.

\subsection*{Order complexes}
Let $P$ be a poset.
The \emph{order complex} $\OC P$ of $P$ is defined to be
the abstract simplicial complex whose faces are the finite chains of $P$.
We denote the geometric realization $\GR{\OC P}$ of the order complex
simply by $\GR P$,
and we sometimes refer to the geometric realization of the order complex of $P$
simply as the order complex of $P$.

\subsection*{Face posets}
Let $K$ be an abstract simplicial complex.
The inclusion defines a partial order on $K$.
The poset $(K, {\subset})$ is called the \emph{face poset} of $K$,
and denoted by $F(K)$.
Usually, the face poset means the induced subposet of $F(K)$
consisting of the non-empty faces.
But our definition, $F(K)$ contains the empty face as a minimum element
unless $K$ itself is empty.

\section{A representation theorem for locally distributive semilattices}
\label{sec:rep}

It is well-known that any distributive lattice of finite height
is isomorphic to the poset of the down sets of a finite poset,
which is known as Birkhoff's representation theorem for distributive lattices
(see \cite[Theorem~107]{Gra}).
In this section, we discuss its extension
for locally distributive semilattices.
The basic idea of this extension can be seen in Section~7.6 of \cite{CCHO}.

Let $S$ be a non-empty locally distributive semilattice
of locally finite height.
An element $x$ of $S$ is \emph{join-reducible}, or simply \emph{reducible},
if there exist $y, z \in S^{< x}$ such that $x = y \vee z$.
An element $x$ of $S$ is \emph{join-irreducible}, or simply \emph{irreducible},
if $x$ is neither reducible nor equal to $0$.
Let $\Irr S$ denote the induced subposet
consisting of the irreducible elements of $S$.

\begin{proposition}
  For $x \in S$, the following are equivalent.
  \begin{itemize}
    \item $x$ is irreducible.
    \item For any finite subset $F$ of $S^{\le x}$,
      $\bigvee F = x$ implies $x \in F$.
  \end{itemize}
\end{proposition}

\begin{proof}
  The proof is done by induction on $\# F$.
\end{proof}

\begin{lemma} \label{lem:bdd->fin}
  Let $A$ be a set of irreducible elements of $S$.
  If $A$ is bounded in $S$, then $A$ is finite.
\end{lemma}

\begin{proof}
  Take $u \in S$ such that $A \subset S^{\le u}$.
  It is enough to show $\# A \le \height_S(u)$.
  Otherwise we can take $n > \height_S(u)$ and $a_1, \dots, a_n \in A$
  such that $a_i \not\ge a_j$ for $i < j$.
  Set $b_i = \bigvee_{j \le i} a_j$ for $i = 0, \dots, n$.
  Clearly, the sequence $b_0, \dots, b_n$ is weakly increasing.
  If the equation $b_{j - 1} = b_j$ holds, then we have
  \begin{equation*}
    a_j
    = b_j \wedge a_j
    = b_{j - 1} \wedge a_j
    = \Bigl( \bigvee_{i < j} a_i \Bigr) \wedge a_j
    = \bigvee_{i < j} (a_i \wedge a_j).
  \end{equation*}
  Since $a_j$ is irreducible, there exists $i < j$
  such that $a_i \wedge a_j = a_j$, that is, $a_i \ge a_j$,
  which contradicts the assumption for $a_1, \dots, a_n$.
  Hence the sequence $b_0, \dots, b_n$ is strictly increasing,
  and thus forms a chain in $S^{\le u}$ of length $n$,
  which contradicts the assumption $n > \height_S(u)$.
\end{proof}

A \emph{down set} of a poset $P$ is a subset $I$ of $P$ such that
$x \le y$ and $y \in I$ imply $x \in I$ for any $x, y \in P$.
Let $\Down P$ denote the set of the down sets of $P$.
For a subset $\sigma$ of $P$, we define
\begin{equation*}
  \overline \sigma = \{\, x \in P \;|\;
  \text{there exists $y \in \sigma$ such that $x \le y$} \,\}.
\end{equation*}
Then $\overline \sigma$ is the smallest down set of $P$
which contains $\sigma$.

Let $K$ be an abstract simplicial complex,
and fix a partial order $\le$ on $V(K)$.
A face of $K$ is said to be a \emph{down face}
if it is a down set of $V(K)$ with respect to this order.
Let $DF(K)$ denote the induced subposet of $F(K)$ consisting of the down faces.
This partial order $\le$ on $V(K)$ is said to be a \emph{compatible order}
on $K$ if any face of $K$ is contained in some down face of $K$.
Equivalently, $\overline \sigma$ is a face of $K$ for any face $\sigma$ of $K$.

\begin{proposition}
  Let $K$ be an abstract simplicial complex, and fix a partial order on $V(K)$.
  Then $F(K)$ is a locally Boolean semilattice of locally finite height,
  and $DF(K)$ is a locally distributive semilattice of locally finite height.
  Moreover, the meets, the joins and the heights in $DF(K)$
  coincide with the restrictions of those in $F(K)$.
\end{proposition}

\begin{proof}
  The meet is given by the intersection,
  the join by the union if exists,
  and the height by the size of a face, which is finite.
\end{proof}

\begin{theorem} \label{thm:LDS}
  Let $S$ be a locally distributive semilattice of locally finite height.
  Then there exist an abstract simplicial complex $K$
  and a compatible order on $K$ such that $DF(K)$ is isomorphic to $S$.
\end{theorem}

\begin{proof}
  Let $K$ be the abstract simplicial complex whose faces are
  the subsets $\sigma$ of $\Irr S$ bounded in $S$.
  The finiteness of faces of $K$ follows from Lemma~\ref{lem:bdd->fin}.
  The induced order on $V(K) = \Irr S$ is a compatible order on $K$,
  since $\sigma \subset S^{\le u}$
  implies $\overline \sigma \subset S^{\le u}$.

  Let $\phi \colon S \to DF(K)$
  and $\psi \colon DF(K) \to S$ be the maps defined by
  \begin{align*}
    \phi(x) &= (\Irr S)^{\le x} && (x \in S) \\
    \psi(\sigma) &= \bigvee \sigma && (\sigma \in DF(K)).
  \end{align*}
  Clearly, $\phi$ and $\psi$ are well-defined and order-preserving.
  We will show these maps are inverses of each other.

  It is clear that $\psi \circ \phi(x) \le x$ holds for any $x \in S$.
  We now show that the equation holds by induction on $\height_S(x)$.
  The case either $x = 0$ or $x \in \Irr S$ is trivial.
  Assume that $x$ is reducible,
  that is, $x = y \vee z$ for some $y, z \in S^{< x}$.
  By the induction hypothesis, we have
  \begin{equation*}
    \psi \circ \phi(x)
    = \psi \circ \phi(y \vee z)
    \ge \psi \circ \phi(y) \vee \psi \circ \phi(z)
    = y \vee z
    = x.
  \end{equation*}

  It is clear that $\phi \circ \psi(\sigma) \supset \sigma$
  for any $\sigma \in DF(K)$.
  For $x \in \phi \circ \psi(\sigma) = (\Irr S)^{\le \bigvee \sigma}$,
  we have
  \begin{equation*}
    x
    = x \wedge \Bigl( \bigvee \sigma \Bigr)
    = \bigvee_{y \in \sigma} (x \wedge y).
  \end{equation*}
  Since $x$ is irreducible, there exists $y \in \sigma$
  such that $x = x \wedge y$, that is, $x \le y$.
  Since $\sigma$ is a down set, $x$ belongs to $\sigma$.
\end{proof}

\begin{corollary} \label{cor:LBS}
  Let $S$ be a locally Boolean semilattice of locally finite height.
  Then there exists an abstract simplicial complex $K$
  such that $F(K)$ is isomorphic to $S$.
\end{corollary}

\begin{proof}
  It is enough to show that $\Irr S$ is an antichain,
  that is, there is no non-trivial ordering.
  Let $x, y \in \Irr S$ with $x > y$.
  Since $S^{\le x}$ is Boolean, there exists $z \in S^{\le x}$
  such that $y \wedge z = 0$ and $y \vee z = x$.
  Since $x$ is irreducible and $y < x$, we have $z = x$.
  Thus we have
  \begin{equation*}
    0 = y \wedge z = y \wedge x = y,
  \end{equation*}
  which contradicts to the assumption that $y$ is irreducible.
\end{proof}

\begin{proposition}
  Let $K$ be an abstract simplicial complex, and fix a compatible order on $K$.
  Then the following are equivalent.
  \begin{enumerate}
    \item $K$ is a flag complex.
    \item $F(K)$ is a flag semilattice.
    \item $DF(K)$ is a flag semilattice.
  \end{enumerate}
\end{proposition}

\begin{proof}
  $(1) \Leftrightarrow (2)$ and $(2) \Rightarrow (3)$ are trivial.
  We now show $(3) \Rightarrow (2)$.
  Assume that $DF(K)$ is a flag semilattice.
  Let $\sigma_1, \sigma_2, \sigma_3$ be pairwise bounded elements of $F(K)$.
  Then $\overline \sigma_1, \overline \sigma_2, \overline \sigma_3$ are
  pairwise bounded in $DF(K)$,
  since an upper bound of $\overline \sigma_i$ and $\overline \sigma_j$
  is given by $\overline{\sigma_i \cup \sigma_j}$.
  Thus there exists an upper bound of
  $\{ \overline \sigma_1, \overline \sigma_2, \overline \sigma_3 \}$
  in $DF(K)$, which is also an upper bound of
  $\{ \sigma_1, \sigma_2, \sigma_3 \}$ in $F(K)$.
\end{proof}

\section{Metric spaces} \label{sec:met}

The metric on $\R^n$ defined by $d(x, y) = \sqrt{\sum_{i = 1}^n (x_i - y_i)^2}$
is called the \emph{Euclidean metric},
and $\R^n$ with the Euclidean metric is called the \emph{Euclidean space},
which will be denoted by $\E^n$.
A metric space $X$ is said to be \emph{complete}
if any Cauchy sequence in $X$ converges.
For a metric space $(X, d_X)$ and a subset $A$ of $X$,
the restriction of $d_X$ on $A \times A$ is called the \emph{induced metric}
on $A$.
For two metric spaces $X$ and $Y$, a map $f \colon X \to Y$ is
\emph{non-expanding} if $d_Y(f(x), f(x')) \le d_X(x, x')$ holds
for any $x, x' \in X$,
and $f$ is \emph{distance preserving} if $d_Y(f(x), f(x')) = d_X(x, x')$ holds
for any $x, x' \in X$.
Clearly, a distance-preserving map is injective.
A bijective distance-preserving map is called an \emph{isometry}.
Two metric spaces are called \emph{isometric}
if an isometry between them exists.
For a metric space $X$ and $x, y \in X$,
a \emph{geodesic path} from $x$ to $y$ in $X$ is a distance-preserving map
$\gamma \colon [0, \ell] \to X$ which sends the endpoints $0$ and $\ell$
to $x$ and $y$, respectively.
Here $[0, \ell]$ denotes the closed interval of $\R$ with the standard metric,
that is, $d_{[0, \ell]}(s, t) = \lvert s - t \rvert$.
In this case, we have $d_X(x, y) = \ell$.
A metric space $X$ is sait to be \emph{geodesic} if for any $x, y \in X$
there exists a geodesic path from $x$ to $y$ in $X$.

\subsection*{CAT(0) properties}
A geodesic metric space $X$ is \emph{CAT(0)} if for any $x, y, z \in X$
and any geodesic path $\gamma \colon [0, \ell] \to X$ from $x$ to $y$ in $X$,
the inequality
\begin{equation*}
  d_X(\gamma(t \ell), z)^2
  \le t \cdot d_X(y, z)^2
  + (1 - t) \cdot d_X(x, z)^2
  - t (1 - t) \cdot d_X(x, y)^2
\end{equation*}
holds for all $t \in [0, 1]$.
Roughly speaking, this inequality means that the any triangle in $X$
whose edges are geodesic paths is at least as thin as the comparison triangle
of the same side lengths in the Euclidean space.
We say a metric space $X$ has \emph{non-positive curvature},
or is \emph{locally CAT(0)},
if for any $x \in X$ there exists $r > 0$ such that
the $r$-open ball $\{\, y \in X \;|\; d_X(x, y) < r \,\}$
around $x$ is a CAT(0) geodesic space with the induced metric.

Let us note that if a geodesic space $X$ is CAT(0),
then $X$ is uniquely geodesic, that is, for any pair of points of $X$
there uniquely exists a geodesic path between them.
Since the unique geodesic path can be taken continuously with respect to
the end points, any non-empty CAT(0) geodesic space must be contractible.

\begin{theorem}[The Cartan-Hadamard Theorem {\cite[II.4.1(2)]{BH}}]
  \label{thm:CH}
  Let $X$ be a complete metric space.
  Then the following are equivalent.
  \begin{enumerate}
    \item $X$ has non-positive curvature, and is simply connected.
    \item $X$ is a CAT(0) geodesic space.
  \end{enumerate}
\end{theorem}

\subsection*{Euclidean polyhedral complexes}
In this subsection, we review the definition and basic properties of
Euclidean polyhedral complexes.
Roughly speaking, Euclidean polyhedral complexes are obtained from
Euclidean polytopes by gluing them along isometric faces.
We interest in conditions for Euclidean polyhedral complexes
to have (locally) CAT(0) metric.

A \emph{Euclidean polytope} is a polytope in a Euclidean space
with the induced metric.
A \emph{Euclidean polyhedral complex} is a set $X$ equipped with
a family $\{ (P_\lambda, i_\lambda) \}_{\lambda \in \Lambda}$ of pairs of
a Euclidean polytope $P_\lambda$
and an injection $i_\lambda \colon P_\lambda \to X$
which satisfies the following:
\begin{itemize}
  \item The images of $i_\lambda$ cover $X$, that is,
    $X = \bigcup_{\lambda \in \Lambda} i_\lambda(P_\lambda)$.
  \item Let $\lambda, \lambda' \in \Lambda$ such that
    $i_\lambda(P_\Lambda) \cap i_{\lambda'}(P_{\lambda'}) \neq \emptyset$.
    Then the inverse image of the intersection under $i_\lambda$
    is a face of $P_\lambda$,
    similarly the inverse image under $i_{\lambda'}$
    is a face of $P_{\lambda'}$,
    and the induced bijection
    \begin{equation*}
      i_{\lambda'}^{-1} \circ i_\lambda \colon
      i_\lambda^{-1}(i_\lambda(P_\lambda) \cap i_{\lambda'}(P_{\lambda'}))
      \to
      i_{\lambda'}^{-1}(i_\lambda(P_\lambda) \cap i_{\lambda'}(P_{\lambda'}))
    \end{equation*}
    is an isometry with respect to the induced metrics.
\end{itemize}

The maps $\{ i_\lambda \}_{\lambda \in \Lambda}$ are called
\emph{face maps} of $X$,
and their images are called \emph{faces} of $X$.
The restriction of $i_\lambda$ on a face of $P_\lambda$
is also called a face map of $X$,
and its image is also called a face of $X$.

By definition, our Euclidean polyhedral complexes are regular,
that is, all face maps are injective.
Moreover, our Euclidean polyhedral complexes are simple, that is,
any two faces intersect in at most one face of them.

For $x, y \in X$,
a \emph{string} from $x$ to $y$ in $X$ is a finite sequence
$\Sigma = \{ (\lambda_i, x_i, y_i) \}_{i = 1}^m$
of triples which satisfy the following.
\begin{itemize}
  \item $\lambda_i \in \Lambda_i$ for $i = 1, \dots, m$
  \item $x_i, y_i \in P_{\lambda_i}$ for $i = 1, \dots, m$
  \item $x = i_{\lambda_1}(x_1)$
  \item $i_{\lambda_i}(y_i) = i_{\lambda_{i + 1}}(x_{i + 1})$
    for $i = 1, \dots, m - 1$
  \item $i_{\lambda_m}(y_m) = y$
\end{itemize}
The \emph{length} of a string $\Sigma = \{ (\lambda_i, x_i, y_i) \}_{i = 1}^m$
is defined by
\begin{equation*}
  \ell(\Sigma) = \sum_{i = 1}^m d_{P_{\lambda_i}}(x_i, y_i).
\end{equation*}
The \emph{intrinsic pseudo-metric} on $X$ is defined by
\begin{equation*}
  d_X(x, y) = \inf \{\, \ell(\Sigma) \;|\;
  \text{$\Sigma$ is a string from $x$ to $y$ in $X$} \,\}.
\end{equation*}
If there is no string from $x$ to $y$ in $X$,
we define $d_X(x, y) = \infty$.
The intrinsic pseudo-metric can be characterized as follows:
for any pseudo-metric space $Z$ and any map $f \colon X \to Z$,
$f$ is non-expanding if and only if
$f \circ i_\lambda \colon P_\lambda \to Z$ is non-expanding
for each $\lambda \in \Lambda$.
Equivalently, the intrinsic pseudo-metric is the largest pseudo-metric
such that all face maps $i_\lambda \colon P_\lambda \to X$ are non-expanding.
Let us note that any string $\Sigma = \{ (\lambda_i, x_i, y_i) \}_{i = 1}^m$
induces a path in $X$ by concatenating the line segment $[x_i, y_i]$
in $P_{\lambda_i}$.
We say $X$ is \emph{connected} if any pair of points of $X$
can be connected by a string in $X$.
We say that $X$ has \emph{finite shapes} if the number of isometry types
of $\{\, P_\lambda \;|\; \lambda \in \Lambda \,\}$ is finite.
Bridson showed the following.

\begin{theorem}[{\cite[Chapter I.7]{BH}}] \label{thm:FS}
  If $X$ is a connected Euclidean polyhedral complex of finite shapes,
  then the intrinsic pseudo-metric is a metric,
  and $X$ is a complete geodesic metric space.
  Moreover, any geodesic path in $X$ is obtained from a string.
\end{theorem}

A \emph{cubical complex} is a Euclidean polyhedral complex $X$
such that each face of $X$ is isometric to a unit cube
$I^n = [0, 1]^n \subset \E^n$.
Note that a cubical complex has finite shapes
if and only if it has finite dimension.
A face of $X$ isometric to $I^0$ is called a \emph{vertex} of $X$.
Since a vertex $v$ of $X$ is a one-point subspace of $X$,
we identify $v$ as an element of $X$.
A face of $X$ isometric to $I^1$ is called an \emph{edge} of $X$.
Two distinct vertex $v$ and $w$ of $X$ is said to be \emph{adjacent}
if there exists an edge of $X$ which contains both $v$ and $w$.
For a vertex $v$ of $X$,
the \emph{(cubical) link} $\lk(v; X)$ of $v$ in $X$
is defined to be the abstract simplicial complex whose faces are
the finite subsets $\sigma$ of vertices adjacent to $v$ such that
there exist a face of $X$ containing $v$ and $\sigma$.

Gromov showed the following:

\begin{theorem}[\cite{Gro}] \label{thm:NPC-flag}
  Let $X$ be a finite-dimensional cubical complex.
  Then $X$ has non-positive curvature
  if and only if $\lk(v; X)$ is a flag complex for any vertex $v$ of $X$.
\end{theorem}

We will discuss a translation and an extension of this characterization
by Gromov.
In order to do this, we now introduce a notion of cubical cone,
which behaves as a partial inverse of taking the cubical links.
Let $K$ be a finite-dimensional abstract simplicial complex $K$.
The \emph{cubical cone} $\CC(K)$ of $K$ is defined to be the cubical complex
such that
\begin{equation*}
  \CC(K) = \bigcup_{\sigma \in K} I^\sigma \subset \E^{(V(K))},
\end{equation*}
where
\begin{equation*}
  I^\sigma = \Bigl\{\, \sum_{v \in \sigma} t_v v \;\Bigm|\;
  t_v \in [0, 1] \,\Bigr\}
  \subset \E^{(\sigma)} \subset \E^{(V(K))}.
\end{equation*}
Here $\E^{(A)}$ denotes the direct sum $\bigoplus_{a \in A} \R a$
with the Euclidean metric with respect to $A$,
that is, $d_{\E^{(A)}}(\sum_a t_a a, \sum_a s_a a)
= \sqrt{\sum_a (t_a - s_a)^2}$.
Face maps of $\CC(K)$ are inclusions $I^\sigma \hookrightarrow \CC(K)$
for $\sigma \in K$.
Here we see $I^\sigma$ as a Euclidean polytope in $\E^{(\sigma)}$,
which is isometric to the $\# \sigma$-dimensional unit cube.

\begin{proposition}
  A vertex of $\CC(K)$ has form $\chi_\sigma$ for $\sigma \in F(K)$,
  where $\chi_\sigma$ denotes $\sum_{v \in \sigma} v$.
  Moreover, the link $\lk(\chi_\sigma; \CC(K))$
  is isomorphic to $\widetilde \sigma * \lk(\sigma; K)$.
\end{proposition}

\begin{proof}
  The first assertion is obvious.
  We now show the second.
  A vertex adjacent to $\chi_\sigma$ has form either
  $\chi_{\sigma \setminus \{ v \}}$ for $v \in \sigma$ or
  $\chi_{\sigma \cup \{ w \}}$ for $w \in V(\lk(\sigma; K))$.
  The obvious bijection
  \begin{equation*}
    V(\widetilde \sigma * \lk(\sigma; K))
    = V(\widetilde \sigma) \sqcup V(\lk(\sigma; K))
    \to V(\lk(\chi_\sigma; \CC(K)))
  \end{equation*}
  gives an isomorphism between abstract simplicial complexes.
\end{proof}

\begin{proposition} \label{prp:CC}
  The cubical cone $\CC(K)$ is a CAT(0) space if and only if
  $K$ is a flag complex.
\end{proposition}

\begin{proof}
  By Theorem~\ref{thm:FS}, $\CC(K)$ is a complete metric space.
  Since $\CC(K)$ is star-shaped at the origin, $\CC(K)$ is contractible,
  and thus simply connected.
  By Theorem~\ref{thm:CH},
  $\CC(K)$ is CAT(0) if and only if it has non-positive curvature.
  By Theorem~\ref{thm:NPC-flag}, this is equivalent to
  that the each vertex link of $\CC(K)$ is a flag complex.
  Combining the previous proposition and Proposition~\ref{prp:flag},
  we have the assertion.
\end{proof}

\begin{proposition}
  The inclusion $i \colon \CC(K) \to \E^{(V(K))}$ is a non-expanding map.
  Moreover, for $\xi, \eta \in \CC(K)$, if the equation
  \begin{equation*}
    d_{\CC(K)}(\xi, \eta) = d_{\E^{(V(K))}}(\xi, \eta)
  \end{equation*}
  holds, then the line segment between $\xi$ and $\eta$ in $E^{(V(K))}$
  is contained in $\CC(K)$.
\end{proposition}

\begin{proof}
  The first assertion follows from that
  the composition $I^\sigma \hookrightarrow \CC(K) \hookrightarrow \E^{(V(K))}$
  is distance preserving for $\sigma \in F(K)$.
  We now show the second.
  Let $\xi, \eta \in \CC(K)$ such that
  $d_{\CC(K)}(\xi, \eta) = d_{\E^{(V(K))}}(\xi, \eta)$.
  Take a geodesic path
  $\gamma \colon [0, \ell] \to \CC(K)$ from $\xi$ to $\eta$ in $\CC(K)$.
  Then $i \circ \gamma$ is a geodesic path from $\xi$ to $\eta$
  in the Euclidean space $\E^{(V(K))}$, which implies the assertion.
\end{proof}

\section{Orthoscheme complex} \label{sec:ortho}

In this section, we consider the orthoscheme complex of a poset,
which is the order complex equipped with
a certain Euclidean polyhedral complex structure.

For positive real numbers $\ell_1, \dots, \ell_d$, the \emph{orthoscheme}
$O(\ell_1, \dots, \ell_d)$ is defined to be the Euclidean polytope in $\E^d$
spanned by $v_i = \sum_{j = 1}^i \ell_j e_j$ for $i = 0, \dots, d$.
Here $e_1, \dots, e_d$ denote the standard orthonormal basis of $\E^d$.
Then the orthoscheme  $O(\ell_1, \dots, \ell_d)$ is a $d$-dimensional
Euclidean simplex satisfying the following properties:
\begin{itemize}
  \item The edge $v_i v_j$ is orthogonal to $v_j v_k$
    for $0 \le i \le j \le k \le d$.
  \item The edge $v_{i - 1} v_i$ has length $\ell_i$
    for $1 \le i \le d$.
  \item The edge $v_i v_j$ has length $\sqrt{\sum_{k = i + 1}^j \ell_k^2}$
    for $0 \le i \le j \le d$.
\end{itemize}
Let us note that the unit $d$-orthoscheme $O(1, \dots, 1)$ is isometric to
the facet of the barycentric subdivision of the cube $[-1, 1]^d$.

Let $P$ be a poset, and $h \colon P \to \R$ be a strictly order-preserving map.
We now construct a Euclidean polyhedral complex structure
on the order complex $\GR P$ by using $h$.
For a finite chain $\sigma = \{ x_0 < \dots < x_d \}$ of $P$,
Let us define
\begin{equation*}
  O^\sigma = O(\sqrt{h(x_1) - h(x_0)}, \sqrt{h(x_2) - h(x_1)},
  \dots, \sqrt{h(x_d) - h(x_{d - 1})}),
\end{equation*}
and define $i_\sigma \colon O_\sigma \to \GR P$ to be the affine map
which sends $v_i$ to $x_i$ for $i = 0, \dots, d$.
Then $i_\sigma$ is an injection onto $\Delta^\sigma$.
We can see that $i_\sigma$ gives a Euclidean metric on $\Delta^\sigma$
such that $d_{\Delta^\sigma}(x_i, x_j) = \sqrt{h(x_j) - h(x_i)}$
for $0 \le i \le j \le d$.
The orthoscheme complex of $P$ with respect to $h$
is defined to be the Euclidean polyhedral complex
on the geometric realization $\GR P$
whose face maps are $i_\sigma \colon O^\sigma \to \GR P$
for $\sigma \in \OC P$.
We say that a poset $P$ is \emph{connected} if for any $x, y \in P$
there exists a finite sequence $x_0, \dots, x_{2 n}$ in $P$ such that
\begin{equation*}
  x = x_0 \le x_1 \ge x_2 \le \dots \ge x_{2 n} = y.
\end{equation*}
Let us not that $P$ is connected if and only if
the orthoscheme complex $\GR P$ is connected.
By using Theorem~\ref{thm:FS} we have the following.

\begin{proposition}
  If $P$ is connected and the image of $h \colon P \to \R$ is finite,
  then the orthoscheme complex $\GR P$ of $P$ with respect to $h$ is
  a complete geodesic metric space.
\end{proposition}

In the rest of this paper,
we treat only posets of finite height,
and discuss their orthoscheme complexes
with respect to the canonical height function
\begin{equation*}
  \height_P \colon P \to \{ 0, 1, \dots, \height P \} \subset \R.
\end{equation*}

\begin{lemma} \label{lem:FK-CCK}
  Let $K$ be a finite-dimensional abstract simplicial complex.
  Then the orthoscheme complex $\GR{F(K)}$
  is isometric to the cubical cone $\CC(K)$.
\end{lemma}

\begin{proof}
  Define $\phi \colon \GR{F(K)} \to \CC(K)$ by
  \begin{equation*}
    \phi \Bigl( \sum_{i = 0}^d t_i \sigma_i \Bigr)
    = \sum_{i = 0}^d t_i \chi_{\sigma_i}.
  \end{equation*}
  To show that $\phi$ is a bijection, we now construct the inverse $\psi$
  of $\phi$.
  Let $\xi = \sum_{v \in V(K)} t_v v$ be an element of $\CC(K)$.
  Note that $\{\, v \in V(K) \;|\; t_v > 0 \,\}$ is finite and
  forms a face of $K$.
  Take a descending sequence $1 = s_0 > s_1 > \dots > s_{d + 1} = 0$ such that
  \begin{equation*}
    \{\, t_v \;|\; v \in V(K) \,\} \cup \{ 0, 1 \}
    = \{ s_0, s_1, \dots, s_{d + 1} \},
  \end{equation*}
  and let
  \begin{equation*}
    \sigma_i = \{\, v \in V(K) \;|\; t_v \ge s_i \}
  \end{equation*}
  for $i = 0, \dots, d$.
  Then $\sigma_0 \subsetneq \sigma_1 \subsetneq \dots \subsetneq \sigma_d$
  is a finite chain of $F(K)$.
  We define
  \begin{equation*}
    \psi(\xi) = \sum_{i = 0}^d (s_i - s_{i + 1}) \sigma_i.
  \end{equation*}
  We can easily check that $\psi$ is the inverse of $\phi$.

  We next show that $\phi$ is distance preserving.
  By definition we can check that the restriction of $\phi$ on $\Delta^\Sigma$
  is distance preserving for any finite chain $\Sigma$ of $F(K)$.
  Using the characterization of intrinsic metric to $\GR{F(K)}$,
  it follows that $\phi$ is non-expanding.
  Moreover, a string in $\CC(K)$ can be decomposed via $\psi$ into a string
  in $F(K)$ of the same length,
  which implies that $\phi$ is distance preserving.
\end{proof}

The following gives a translation of Theorem~\ref{thm:NPC-flag}
for orthoscheme complex.

\begin{theorem} \label{thm:LBS-CAT0}
  Let $S$ be a locally Boolean semilattice of finite height.
  Then the orthoscheme complex $\GR S$ is a CAT(0) space
  if and only if $S$ is a flag semilattice.
\end{theorem}

\begin{proof}
  By Corollary~\ref{cor:LBS}, we can assume $S = F(K)$
  for some finite-dimensional abstract simplicial complex $K$.
  Note that $F(K)$ is a flag semilattice if and only if $K$ is a flag complex.
  The assertion follows from Proposition~\ref{prp:CC} and the previous lemma.
\end{proof}

The following is an extension of the previous theorem.

\begin{theorem} \label{thm:LDS-CAT0}
  Let $S$ be a locally distributive semilattice of finite height.
  Then the orthoscheme complex $\GR S$ is a CAT(0) space
  if and only if $S$ is a flag semilattice.
\end{theorem}

To show this theorem, we first show the following.

\begin{lemma}
  Let $K$ be a finite-dimensional abstract simplicial complex,
  and fix a compatible order on $K$.
  Then $\GR{DF(K)}$ is a convex subset of $\GR{F(K)}$,
  that is, any geodesic path in $\GR{F(K)}$ between points
  of $\GR{DF(K)}$ is contained in $\GR{DF(K)}$.
  In particular, the induced metric on $\GR{DF(K)}$ from $\GR{F(K)}$
  coincides with the intrinsic metric of its own.
\end{lemma}

\begin{proof}
  Let $\phi \colon \GR{F(K)} \to \CC(K)$ be the isometry defined
  in the proof of Lemma~\ref{lem:FK-CCK}.
  Let $X$ be the image of $\GR{DF(K)}$ under $\phi$.
  By the definition of $\phi$ and the construction of its inverse,
  we have
  \begin{equation*}
    X = \Bigl\{\, \sum_{u \in V(K)} t_u u \in \CC(K) \;\Bigm|\;
    \text{$t_v \ge t_w$ for $v \le w$ in $V(K)$} \,\Bigr\}.
  \end{equation*}
  It is enough to show that $X$ is convex in $\CC(K)$.
  Set
  \begin{equation*}
    \widetilde Y_{vw} = \Bigl\{\, \sum_{u \in V(K)} t_u u \in \E^{(V(K))}
    \;\Bigm|\; t_v \ge t_w \,\Bigr\}
  \end{equation*}
  and $Y_{vw} = \CC(K) \cap \widetilde Y_{vw}$ for $v < w$ in $V(K)$.
  Then we have $X = \bigcap_{v < w} Y_{vw}$.
  Thus it is enough to show that $Y_{vw}$ is convex in $\CC(K)$ for $v < w$.
  We now define $\widetilde \psi_{vw} \colon \E^{(V(K))} \to \widetilde Y_{vw}$
  as follows. For any $\xi \in \E^{(V(K))}$,
  define $\widetilde \psi_{vw}(\xi) \in \widetilde Y_{vw}$
  to be the unique point such that
  \begin{equation*}
    d_{\E^{(V(K))}}(\xi, \widetilde \psi_{vw}(\xi))
    = \inf_{\eta \in \widetilde Y_{vw}} d_{\E^{(V(K))}}(\xi, \eta).
  \end{equation*}
  Indeed, $\widetilde \psi_{vw}$ is given by
  \begin{equation*}
    \widetilde \psi_{vw} \Bigl( \sum_{u \in V(K)} t_u u \Bigr) =
    \max \Bigl\{ t_v, \frac{t_v + t_w}{2} \Bigr\} v
    + \min \Bigl\{ t_w, \frac{t_v + t_w}{2} \Bigr\} w
    + \sum_{u \neq v, w} t_u u
  \end{equation*}
  Let us note that if $\sigma$ is a face of $K$ and $w \in \sigma$,
  then $\sigma \cup \{ v \} \subset \overline \sigma$ is also a face of $K$.
  Thus the image of $\CC(K)$ under $\widetilde \psi_{vw}$ is contained in
  $\CC(K)$.
  We define $\psi_{vw} \colon \CC(K) \to Y_{vw}$ to be the restriction
  of $\widetilde \psi_{vw}$.
  Since $\widetilde \psi_{vw}$ is non-expanding, $\psi_{vw}$ is non-expanding
  on each face, and thus on entire $\CC(K)$.
  Moreover, $\psi_{vw}$ is a retraction.
  Assume that $Y_{vw}$ is not a convex subset of $\CC(K)$.
  Then there exist $\xi, \eta, \zeta \in \CC(K)$ such that
  $\xi, \zeta \in Y_{vw}$, $\eta \notin Y_{vw}$, and
  $d_{\CC(K)}(\xi, \eta) + d_{\CC(K)}(\eta, \zeta) = d_{\CC(K)}(\xi, \zeta)$.
  Take a shortest string $\Sigma = \{ (\sigma_i, x_i, y_i) \}_{i = 1}^m$
  from $\xi$ to $\eta$.
  Then there exists $i = 1, \dots, m$ such that $x_i \in Y_{vw}$
  but $y_i \notin Y_{vw}$.
  For such $x_i$ and $y_i$, $\psi_{vw}$ strictly shortens their distance.
  Thus the resulting string $\Sigma' = \{ (\overline \sigma_i,
  \psi_{vw}(x_i), \psi_{vw}(y_i)) \}_{i = 1}^m$ from $\xi$ to $\psi_{vw}(\eta)$
  has length less than that of $\Sigma$.
  Hence we have
  \begin{align*}
    d_{\CC(K)}(\xi, \zeta)
    &\le d_{\CC(K)}(\xi, \psi_{vw}(\eta)) + d_{\CC(K)}(\psi_{vw}(\eta), \zeta)
    \\&< d_{\CC(K)}(\xi, \eta) + d_{\CC(K)}(\eta, \zeta)
    \\&= d_{\CC(K)}(\xi, \zeta),
  \end{align*}
  which is a contradiction.
\end{proof}

\begin{proof}[Proof of Theorem~\ref{thm:LDS-CAT0}]
  By Theorem~\ref{thm:LDS}, we can assume $S = DF(K)$
  for some finite-dimensional abstract simplicial complex $K$
  with a fixed compatible order on $K$.
  If $DF(K)$ is a flag semilattice, then $K$ is a flag complex,
  and thus $\GR{F(K)} \cong \CC(K)$ is a CAT(0) space.
  Hence its convex subset $\GR{DF(K)}$ is also a CAT(0) space.

  We now show the converse.
  Assume that $\GR{DF(K)}$ is a CAT(0) space.
  Let $\sigma_i \ (i = 1, 2, 3)$ be pairwise bounded elements of $DF(K)$,
  that is, $\sigma_i \cup \sigma_j \in DF(K)$ for $i, j \in \{ 1, 2, 3 \}$.
  Let $X$ be the image of $\GR{DF(K)} \subset \GR{F(K)}$
  under the isometry $\phi \colon \GR{F(K)} \to \CC(K)$
  in the proof of Lemma~\ref{lem:FK-CCK}.
  Then $X$ is isometric to $\GR{DF(K)}$ with the induced metric from $\CC(K)$,
  and thus $X$ is a CAT(0) space.
  Since the line segment $[\chi_{\sigma_i}, \chi_{\sigma_j}]$
  in $\E^{(V(K))}$ is contained in $X$, we have
  \begin{equation*}
    d_X(\chi_{\sigma_i}, \chi_{\sigma_j})
    = d_{\E^{(V(K))}}(\chi_{\sigma_i}, \chi_{\sigma_j})
  \end{equation*}
  for $i, j \in \{ 1, 2, 3 \}$.
  By the CAT(0) inequality, we have
  \begin{equation*}
    d_X(\frac 1 2 (\chi_{\sigma_1} + \chi_{\sigma_2}), \chi_{\sigma_3})
    \le d_{\E^{(V(K))}}(\frac 1 2 (\chi_{\sigma_1} + \chi_{\sigma_2}),
    \chi_{\sigma_3}).
  \end{equation*}
  This is possible only if the line segment
  $[ \frac 1 2 (\chi_{\sigma_1} + \chi_{\sigma_2}), \chi_{\sigma_3}]$
  is contained in $X$,
  which implies $\sigma_1 \cup \sigma_2 \cup \sigma_3 \in F(K)$,
  and thus $\sigma_1 \cup \sigma_2 \cup \sigma_3 \in DF(K)$.
\end{proof}

\section*{Acknowledgements}
The author would like to thank his supervisor Prof.\ Toshitake Kohno
for helpful advice and support.
This work is supported by JSPS KAKENHI Grant Number 26-3055.

\end{document}